\documentclass[USenglish]{amsart}
\usepackage[pagewise]{lineno}

\usepackage{amssymb,amsmath,amsthm,amscd,epsf,latexsym,verbatim,graphicx,amsfonts,hyperref,epstopdf,xcolor,wasysym}
\usepackage{cancel}
\input epsf.tex
\usepackage{enumerate}
\usepackage[utf8]{inputenc}
\usepackage[T1]{fontenc}
\usepackage{graphicx}
\usepackage{palatino, url, multicol}
\usepackage{subcaption}
\usepackage[margin=1.5in]{geometry}                
\usepackage{tikz}
\usetikzlibrary{patterns}
\usepackage{mathtools}
\usepackage[ruled, vlined]{algorithm2e}

\usepackage{caption} 
\DeclareCaptionLabelFormat{cont}{#1~#2\alph{ContinuedFloat}}
\usepackage{marginnote}

\theoremstyle{plain}
\newtheorem{theorem}{Theorem}[section]
\newtheorem{conjecture}[theorem]{Conjecture}
\newtheorem{lemma}[theorem]{Lemma}
\newtheorem{proposition}[theorem]{Proposition}
\newtheorem{corollary}[theorem]{Corollary}

\newtheorem{mainthm}{Theorem}

\theoremstyle{definition}

\newtheorem{remark}[theorem]{Remark}

\newtheorem{definition}[theorem]{Definition}

\DeclareMathOperator{\It}{It}

\title{A characterization of Thurston's Master Teapot}
\author{Kathryn Lindsey and Chenxi Wu}

\newif\ifdraft\drafttrue 

\def\pr{{\textrm{Prefix}}}
\def\suf{{\textrm{Suffix}}}
\def\r{{\textrm{Reverse}}}
\def\It{{\textrm{It}}}
\def\U{{\Upsilon_2^{cp}}}
\def\X{{\Xi}}
\def\O{{\Omega_2^{cp}}}
\def\0{{\mathbf 0}}

\begin{document}
\maketitle

\begin{abstract} 
  We prove an explicit characterization of the points in Thurston's Master Teapot.   This description can be implemented algorithmically to test whether a point in $\mathbb{C}\times\mathbb{R}$ belongs to the complement of the Master Teapot. As an application, we show that the intersection of the Master Teapot with the unit cylinder is not symmetrical under reflection through the plane that is the product of the imaginary axis of $\mathbb{C}$ and $\mathbb{R}$.
\end{abstract}

\section{Introduction}

The \emph{Master Teapot}, $\U$, for the family $\mathcal{F}^{cp}_2$ of continuous, unimodal, critically periodic interval self-maps is the set 
\[ \U := \overline{ \{ (z,\lambda) \in \mathbb{C} \times \mathbb{R} \mid \lambda = e^{h_{top}(f)} \textrm{ for some } f \in \mathcal{F}^{cp}_2,z \textrm{ is a Galois conjugate of } \lambda \}}, \] and the Thurston set, $\O$, is its projection to the complex plane, i.e.
\[\O := \overline{ \{z \in \mathbb{C} \mid z \textrm{ is a Galois conjugate of } e^{h_{top}(f)} \textrm{ for some } f \in \mathcal{F}^{cp}\} }.\]
A finite approximation of $\U$ is shown in Figure \ref{f:teapot}.  The Master Teapot and Thurston set have rich geometrical and topological structures that have been investigated in several recent works, including  \cite{TiozzoGaloisConjugates, TiozzoTopologicalEntropy, CalegariKochWalker, thurston, thompson, BrayDavisLindseyWu}.  The main result of this paper is an explicit characterization of $\U$ -- a necessary and sufficient condition for a point to be in $\U$.  This characterization can be algorithmically tested and establishes a new connection between horizontal slices of the Master Teapot and iterated function system theory.  Before stating the results precisely, we introduce some terminology and notation.

\medskip

First, we define words and sequences in the alphabet $\{0, 1\}$:

\begin{definition}\label{def:word_seq} \

  \begin{enumerate}
\item A {\em sequence} $w=w_1w_2\ldots $ is an element in $\{0, 1\}^{\mathbb{N}}$. The {\em shift map} $\sigma : \{0,1\}^{\mathbb{N}} \to \{0,1\}^{\mathbb{N}}$ is defined by removing the first element of a sequence, i.e. $\sigma(w_1w_2w_3\dots):=w_2w_3\dots$.
\item A {\em word} $w=w_1w_2\dots w_n$ is an element in $\{0, 1\}^n$ for some positive integer $n$.   The number $n$ is called the {\em length} of the word $w$ and is denoted by $|w|$.
\item For $n \in \mathbb{N}$, the {\em reverse function} $\r : \{0,1\}^n \to \{0,1\}^n$ is defined as
  \[ \r(w_1w_2\dots w_n):=w_nw_{n-1}\dots w_1\]
\item For $k \in \mathbb{N}$, the \emph{$k$-prefix} of a sequence $w=w_1w_2\dots $ is the word
  \[\pr_k(w):=w_1\dots w_k\]
\item For a word $w=w_1 \dots w_n$ of length $n$ and a natural number $k \leq n$, the \emph{$k$-prefix} and \emph{$k$-suffix} of $w$ are the words 
\begin{align*}
  \pr_k(w) & :=w_1\dots w_k\\
  \suf_k(w) &:=w_{n-k+1}w_{n-k+2}\dots w_n \\
  \end{align*}
\end{enumerate}
\end{definition}

Next, we relate words and sequences with dynamics on $\mathbb{C}$ via the following definitions:
\begin{definition}\label{def:ifs} \
\begin{enumerate}
\item For any $z\in\mathbb{C}$, define maps $f_{0,z},f_{1,z}: \mathbb{C} \to \mathbb{C}$ by  
\[f_{0, z}(x):=zx, \quad f_{1, z}(x):=2-zx.\]
\item For any $w=w_1\dots w_n$ and $z \in \mathbb{C}$, set
  \[F(w,z):=f_{w_n,z}\circ \dots \circ f_{w_1,z}(1)\]
\item For any sequence $w=w_1w_2\dots$ and any $z \in \mathbb{C}$ with $|z|>1$, set
\begin{align*}
  H(w, z)  := & \lim_{n\rightarrow\infty}(-1)^{(\sum_{i=1}^nw_i)}z^{-n}  F(\pr_n(w), z) \\
   = &\lim_{n\rightarrow\infty}(-1)^{(\sum_{i=1}^nw_i)}z^{-n}f_{w_n,z}\circ\ldots \circ f_{w_1, z}(1)
  \end{align*}
\item For any sequence $w=w_1w_2\dots$ and $z \in \mathbb{C}$ with $|z|<1$, set
\begin{align*}
  G(w, z):= & \lim_{n\rightarrow\infty}F(\r(\pr_n(w)), z) \\
  = &\lim_{n\rightarrow\infty}f_{w_1,z}\circ \ldots \circ f_{w_n, z}(1) \\
  \end{align*}
\end{enumerate}
\end{definition}

The following definition contains definitions from \cite{MilnorThurston}: 
\begin{definition}\label{def:order} \
  \begin{enumerate}
\item The {\em cumulative sign} of a word $w=w_1w_2\dots w_n$ is defined as $s(w):=(-1)^{\sum_iw_i}$. 
\item The \emph{twisted lexicographic order} $\le_E$ is a total ordering on the set of sequences, defined as follows: $w<_Ew'$, if and only if there is some $k\in\mathbb{N}$, such that $\pr_{k-1}(w)=\pr_{k-1}(w')$, and $s(\pr_{k-1}(w))(w'_{k}-w_{k})>0$. In other words, $w <_E w'$ if and only if, denoting by $k$ the index of the first letter where $w$ and $w'$ differ, either $w'_k>w_k$ and the common $(k-1)$-prefix has positive cumulative sign, or $w'_k<w_k$ and the common $(k-1)$-prefix has negative cumulative sign.
\item We define the total order $\le_E$ on the set of words of length $n$ exactly the same way as above.
\end{enumerate}
\end{definition}

\begin{definition} \label{def:lambdaitinerary} \ 
\begin{enumerate}
\item Let $\lambda\in (1, 2]$. We call the map $f_{\lambda}:[0,1] \to [0,1]$ given by by 
\[f_\lambda(x)=\begin{cases} \lambda x & x\leq 1/\lambda \\ 2-\lambda x & x>1/\lambda\end{cases}\]
the \emph{$\lambda$-tent map}.  
Let $I_{0, \lambda}=[0, 1/\lambda]$, $I_{1, \lambda}=[1/\lambda, 1]$.
\item  The \emph{ $\lambda$-itinerary}, denoted as $\It_\lambda$, is the minimum (with respect to $\leq_E$) sequence $w$ such that for any $k\geq 0$, $f_\lambda^k(1)\in I_{w_{k+1},\lambda}$.
\end{enumerate}
\end{definition}

 One can easily check that $\It_{\lambda}$ is the itinerary of $1$ under $f_\lambda$ in the convention of Milnor-Thurston kneading theory.

Now we introduce a combinatorial condition on sequences:

\begin{definition} \label{def:improvedlambdasuitability}
  For $\lambda \in (1,2]$, a sequence $w$ is called \emph{$\lambda$-suitable} if for every  $\lambda'\in (\lambda, 2]$, the following conditions hold:
    \begin{enumerate}
    \item    $\r ( \pr_n(w)) \leq_E \pr_n(\It_{\lambda'})$ for all $n \in \mathbb{N}$.
    \item   If $\r(\pr_n(w))  = \pr_n(\It_{\lambda'})$, then the cumulative sign $s(\pr_n(w)) = -1$.  
    \item  If $\It_{\lambda'} = 1 \cdot 0^k \cdot 1 \dots$, $k \in \mathbb{N}$, then $w$ does not contain $k+1$ consecutive $0$s. \newline (That is, if $\It_{\lambda'}$ starts with $1$ followed by $k$ $0$s and then $1$, writing $w$ as $w=w_1w_2\ldots,$ there does not exist $n \in \mathbb{N}$ such that $w_i = 0$ for all $n \leq i \leq n+k$.) 
     \item If $k \in \mathbb{N}$ satisfies $\sqrt{2}\leq \lambda^{2^k}<2$, then $w=\mathfrak{D'}^k(w')$ for some sequence $w'$, where $\mathfrak{D'}$ is the map that replaces $0$ with $11$ and $1$ with $01$, such that for every $\lambda'>\lambda^{2^k}$, if $\It_{\lambda'} = 1 \cdot 0^k \cdot 1 \dots$ then $w'$ does not contain $k+1$ consecutive $0$s.
    \end{enumerate} 
\end{definition}
\begin{remark} Every sequence is (vacuously) $2$-suitable. \end{remark}

\noindent For $\lambda\in (1, 2)$, let $\X_\lambda$ be height-$\lambda$ slice of the Master Teapot $\Upsilon_2$:
\[\X_\lambda:=\{z: (z, \lambda)\in \Upsilon_2\}\]

\noindent We will use the following notation:
\begin{align*}
\mathbb{D}  & := \{z \in \mathbb{C} : |z| <1\},  \textup{ the open unit disk} \\
\overline{\mathbb{D}} & := \{z \in \mathbb{C} : |z| \leq 1\}, \textup{ the closed unit disk}\\
S^1 & := \{z \in \mathbb{C} : |z| = 1\}, \textup{ the unit circle} \\
\mathcal{C} & :=  \overline{\mathbb{D}} \times [1,2], \textup{ the closed ``unit cylinder''} \\
\end{align*}

Our main theorem is:

\begin{mainthm} \label{t:improvedinsidecharacterization}
 For any $\lambda \in (1,2]$, the part of the slice $\X_\lambda$ inside the closed unit disk can be characterized as:
\[\X_\lambda \cap \overline{\mathbb{D}} = S^1\cup \left \{z\in\mathbb{D}: G(w, z)=1\text{ for some }\lambda\text{-suitable sequence }w \right\}.\]
\end{mainthm}

There is a similar characterization for outside the unit disc, which follows directly from results in in \cite{TiozzoGaloisConjugates}:

\begin{mainthm} \label{t:improvedoutsidecharacterization}
 For any $\lambda \in [1,2)$, the part of the slice $\X_\lambda$ outside the unit disk is:
  \[\X_\lambda \setminus \overline{\mathbb{D}}=\left \{z \in \mathbb{C} \setminus \overline{\mathbb{D}}: H(\It_\lambda, z)=0 \right \}.\]
\end{mainthm}

\begin{remark} 
 Theorems \ref{t:improvedinsidecharacterization} and \ref{t:improvedoutsidecharacterization} both provide algorithms to certify that a point is in the complement of $\X_\lambda$.  This is useful since the definition of $\U$ is constructive and involves taking a closure. Section \ref{sec:membership} describes these algorithms.
\end{remark}

\begin{remark} Since the set of $\lambda$-suitable sequences is semicontinuous with $\lambda$ (Lemma \ref{lem:semicont}), 
  Theorem \ref{t:improvedinsidecharacterization} implies that if $1<\lambda<\lambda' \leq 2$, then 
  \[\X_\lambda\cap\overline{\mathbb{D}}\subseteq \X_{\lambda'}\cap\overline{\mathbb{D}},\]
   which is the ``Persistence Theorem'' proved in \cite{BrayDavisLindseyWu}.  (The Persistence Theorem is used to prove Theorem \ref{t:improvedinsidecharacterization}.)
\end{remark}

\begin{remark} \label{rem:holdsfortoplevel}
Tiozzo showed in \cite{TiozzoGaloisConjugates}  that 
$$\Omega_2^{cp} \cap \overline{\mathbb{D}} = S^1 \cup \{z \in \mathbb{D} : G(w,z) = 1 \textrm{ for some sequence } w\},$$ and the Persistence Theorem (\cite{BrayDavisLindseyWu}) shows that $\Omega_2 \cap \overline{\mathbb{D}} = \X_2 \cap \overline{\mathbb{D}}$.  It is also known that the unit cylinder is in the teapot, i.e. $S^1 \times [1,2] \subset \Upsilon_2^{cp}$ (\cite{BrayDavisLindseyWu}). Since every sequence is $2$-suitable, this proves the conclusion of Theorem \ref{t:improvedoutsidecharacterization} for the top level of the teapot, the case $\lambda = 2$. 
\end{remark}

\begin{remark}
Our first step towards proving Theorem \ref{t:improvedinsidecharacterization} is proving Theorem  \ref{t:InsideParryConjugatesDontMatter}, and alternative characterization of slices $\Xi_{\lambda} \cap \overline{\mathbb{D}}$.  A corollary of Theorem  \ref{t:InsideParryConjugatesDontMatter} is that all roots in $\overline{\mathbb{D}}$ of all Parry polynomials coming from admissible words -- even reducible Parry polynomials -- are in the Thurston set $\Omega_2^{cp}$. 
\begin{corollary}
$\Omega_2^{cp} \cap \overline{\mathbb{D}}$ is the closure of the set of all roots in $\overline{\mathbb{D}}$ of all Parry polynomials associated to admissible words. 
\end{corollary} 
\noindent In particular, when using Parry polynomials to plot approximations of $\Omega_2^{cp}$, it is not necessary to check whether the Parry polynomials are irreducible. 
\end{remark}

As an application of Theorem \ref{t:improvedinsidecharacterization}, we will show that:    

\begin{mainthm}\label{t:teapotnotsymmetrical}
The part of the Master Teapot inside the unit cylinder is not symmetrical with respect to reflection across the imaginary axis, i.e. $\U \cap \mathcal{C}$ is not invariant  under the map $(z,\lambda) \mapsto (-z,\lambda)$.
\end{mainthm}
\noindent Since Galois conjugates occur in complex conjugate pairs, it is immediate that $(x+iy,\lambda) \in \U$ if and only if $(x-iy,\lambda) \in \U$.  

Theorem \ref{t:teapotnotsymmetrical} is suprising because the Thurston set, $\O$, which is the projection to $\mathbb{C}$ of $\U$, \emph{is} symmetrical under the map $z \mapsto -z$ (Proposition \ref{p:ThurstonSetSymmetrical}).  However, this asymmetry in the Master Teapot is confined to the slices of heights $\geq \sqrt{2}$; one can prove, via the renormalization procedure described in Section 2.3, that the unit cylinder part of slices of height $< \sqrt{2}$ are symmetrical under reflection across the imaginary axis.

\begin{remark}
Theorem \ref{t:improvedinsidecharacterization} allow us to interpret each slice $\X_\lambda\cap\mathbb{D}$ as an analogy of the Mandelbrot set.  The conclusion of Theorem \ref{t:improvedinsidecharacterization} for the top slice (c.f. Remark \ref{rem:holdsfortoplevel}) allows one to characterize $\X_2$ as the union of $S_1$ and the set of all parameters $z \in \mathbb{D}$ such that the point $1$ is an element of the limit set $\Lambda_z$ associated of the iterated function system generated by $f_{0,z}$ and $f_{1,z}$.  
 Theorem \ref{t:improvedinsidecharacterization} suggests viewing $\X_\lambda\cap\mathbb{D}$ as the set of parameters $z$ for which the point $1$ is an element of the ``limit set'' associated to the ``restricted iterated function system'' generated by $f_{0,z}$ and $f_{1,z}$ in which only the compositions represented by $\lambda$-suitable sequences are allowed. 
  \color{black}
  \end{remark}

 Based on numerical experiments, we propose the following conjectured analogy of the Julia-Mandelbrot correspondence:

\begin{conjecture} \label{conj:asymptoticallysimilar}
  For any complex number $|z|<1$, any $\lambda\in (1, 2]$, $\X_\lambda-z$ is asymptotically similar to the set
  \[J_z=\{G(w, z)-1:w\text{ is }\lambda-\text{suitable}\}.\]
\end{conjecture}
\noindent By these two sets being asymptotically similar, we mean there exists a real number $r>0$ and sequences $(t_n), (t'_n) \in \mathbb{C}$ with $t_n, t_n' \to \infty$ such that, denoting Hausdorff distance by $d_{\textrm{Haus}}$, 
$$\lim_{n \to \infty} d_{\textrm{Haus}} \left( \overline{B_r(0)} \cap (t_n(\X_\lambda-z)),  \overline{B_r(0)} \cap (t_n' J_z) \right) = 0.$$

If the Conjecture \ref{conj:asymptoticallysimilar} is true, or at least true for ``enough'' points $z$, we would also be able to show the following:

\begin{conjecture}\label{conj:disconnected}
  There exists $\lambda\in (1, 2)$ such that $\X_\lambda\cap \overline{\mathbb{D}}$ has infinitely many connected components.
\end{conjecture}

 \noindent Figure \ref{f:highresslice} shows a constructive plot (in black) of the slice $\X_{1.8} \cap \overline{\mathbb{D}}$, while Figure \ref{f:highresslice2} shows (in white) points of $\overline{\mathbb{D}} \setminus \X_{1.8}$.
  Comparison of these images suggests the existence of multiple small connected components in the region $\textrm{Re}(z) < 0$ near the inner boundary of the ``ring.''

The Thurston set $\O$ is known to be path-connected and locally connected (Theorem 1.3 of \cite{TiozzoGaloisConjugates}).  It follows from Theorem \ref{t:improvedoutsidecharacterization} that for many heights $\lambda \in (1,2]$, the part of the slice of height $\lambda$ that is outside the unit cylinder consists of more than one connected component. 

Conjecture \ref{conj:disconnected} could be potentially proven by computation via an effective version of Theorem \ref{t:improvedinsidecharacterization} similar to Proposition \ref{p:effective}. However, a tighter bound than that obtained in Proposition \ref{p:effective} would probably be needed for the computation to be feasible. 

\begin{figure}[h] 
  \centering
    \includegraphics[width=\textwidth]{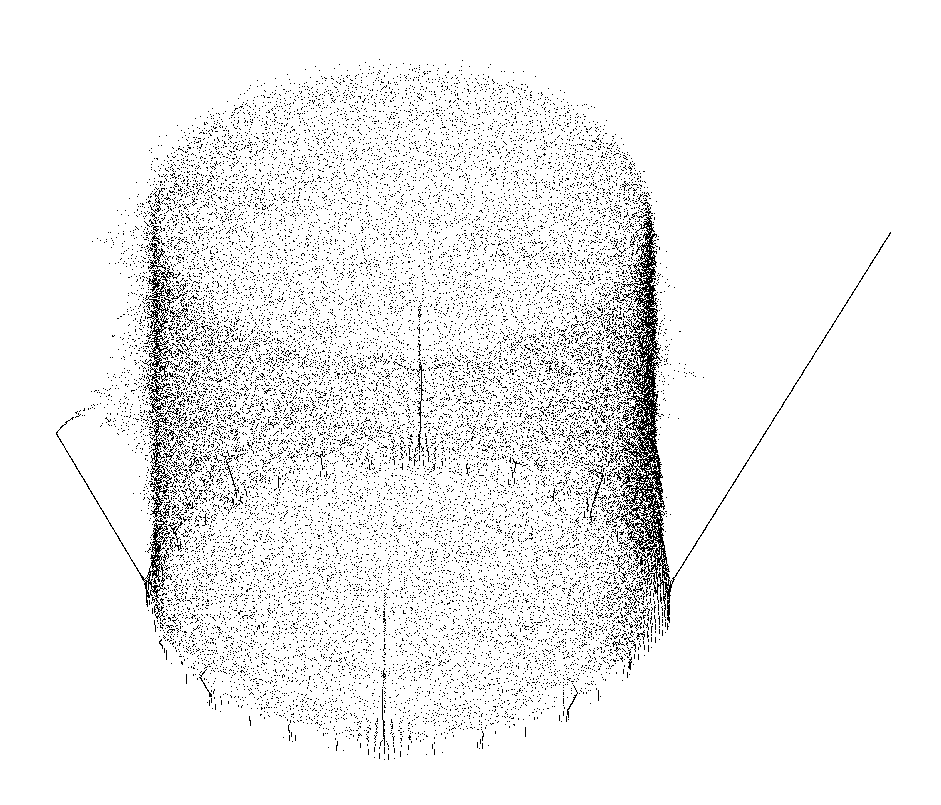}
    \caption{A constructive approximation of the part of $\U$ outside the unit cylinder.   This plot shows
the 56737 points outside the cylinder $S^1 \times [1,2]$ that are roots of the degree 100 partial sums of the kneading power series for $1000$ different growth rates $\lambda$ in $[1,2]$. The "spout" on the right side of the image consists of points of the form $(\lambda,\lambda)$.}
      \label{f:teapot}
\end{figure}

\begin{figure}[h]      
  \centering
    \includegraphics[width=\textwidth]{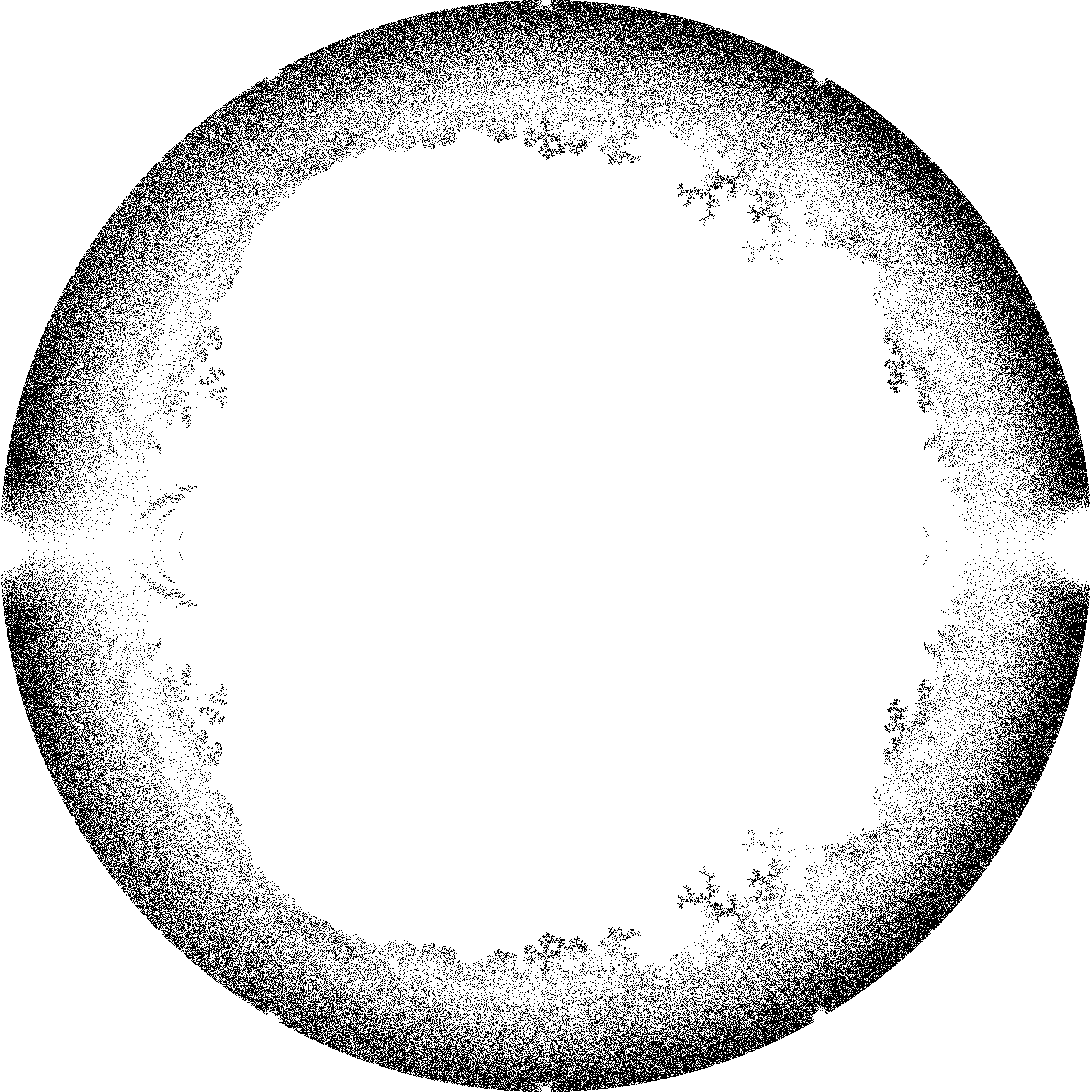}
    \caption{A constructive plot of an approximation of the slice  $\X_{1.8} \cap \mathbb{D}$.  The plotted black points are all the roots of modulus $\leq 1$ of all Parry polynomials for superattracting tent maps with growth rate $<1.8$ and critical length at most $29$.  }
     \label{f:highresslice}
\end{figure}

\begin{figure}[h]      
  \centering
    \includegraphics[width=\textwidth]{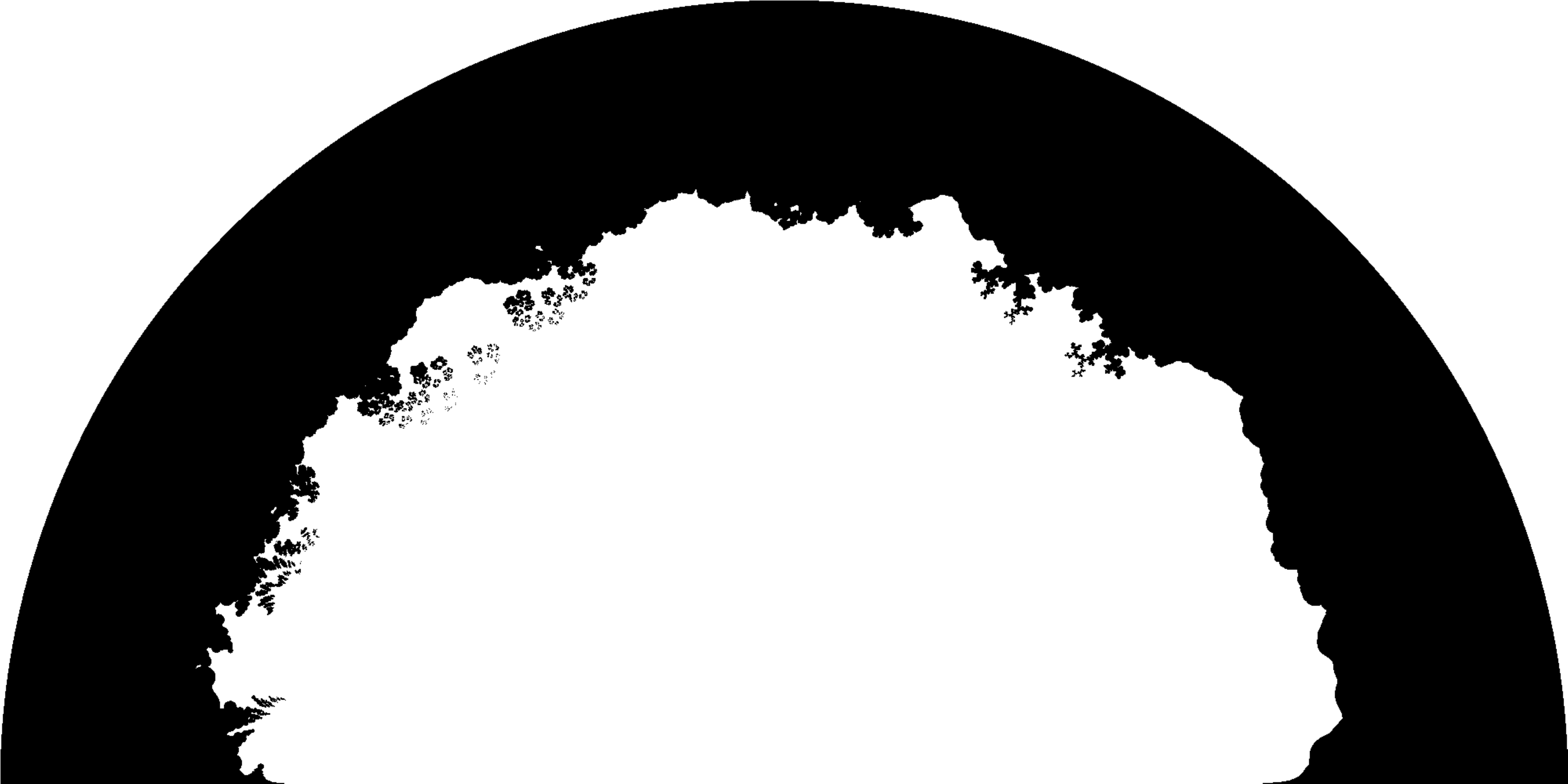}
    \caption{The upper half of the slice $\U \cap (\mathbb{D} \times \{1.8\})$ plotted using Theorem \ref{t:improvedinsidecharacterization}.  Specifically, the plotted white points 
    were shown to be in the complement of $\U$ (by checking the condition of Theorem \ref{t:improvedinsidecharacterization} for all $m \leq 18$). }
     \label{f:highresslice2}
\end{figure}

The structure of the paper is as follows:

\S \ref{s:preliminaries}: \textbf{Preliminaries} provide definitions and notation for Parry polynomials, admissible and dominant words and sequences, growth rates, and the renormalization/doubling operators.
 
\S \ref{sec:renormalizationdoubling}: \textbf{Properties of the doubling map} proves some elementary results about the doubling map which we will need in later sections to extend results about the top part of the teapot to the part with height $<\sqrt{2}$.  

\S \ref{s:reducibleParrypolys}: \textbf{Roots in $\mathbb{D}$ of reducible Parry polynomials} proves Theorem 
 \ref{t:InsideParryConjugatesDontMatter}, which implies that all roots in the unit disk of all Parry polynomials associated to admissible words are in the teapot. 
 
\S  \ref{sec:lambdasuitability}: \textbf{$\lambda$-suitability} discusses $\lambda$-suitability and proves Lemma \ref{lem:main}, which is the
 key combinatorial result we need to prove Theorem~\ref{t:improvedinsidecharacterization}.
 
\S \ref{s:insidecylinder}: \textbf{Characterization inside the unit cylinder} uses Theorem  \ref{t:InsideParryConjugatesDontMatter} and Lemma \ref{lem:main} to prove Theorem \ref{t:improvedinsidecharacterization}.

\S \ref{sec:outsidecharacterization}: \textbf{Characterization outside the unit cylinder} proves Theorem \ref{t:improvedoutsidecharacterization}.

\S  \ref{sec:membership}: \textbf{Algorithms to test membership of $\X_\lambda$} presents algorithms, derived from Theorems 
 \ref{t:improvedinsidecharacterization} and \ref{t:improvedoutsidecharacterization}, which will detect if a point $(z,\lambda) \in \mathbb{C} \times \mathbb{R}$ belongs to the complement of the height-$\lambda$ slice $\X_{\lambda}$, and proves lemmas that justify the algorithms. 
 
 \S \ref{sec:asymmetry}: \textbf{Asymmetry} proves Theorem \ref{t:teapotnotsymmetrical} by exhibiting a point $(z,\lambda)$ that is in the teapot and using the algorithm from \S \ref{sec:membership} to prove that $(-\bar{z},\lambda)$ is in  the complement of the slice $\X_{\lambda}$.

\subsection*{Acknowledgements} The authors thank Diana Davis for many helpful conversations.  Kathryn Lindsey was supported by the National Science Foundation under grant  DMS-1901247. 

\section{Preliminaries} \label{s:preliminaries}

\subsection{Concatenation} We use $\cdot$ or just adjacency to denote concatenations, i.e. for any word $w=w_1\ldots w_n$ and any word or sequence $v=v_1v_2\ldots$, $$w\cdot v = wv = w_1\ldots w_n v_1 v_2 \ldots.$$ We denote the concatenation of $n$ copies of a word $w$ by $w^n$, for $n \in \mathbb{N} \cup \{\infty\}$.

\subsection{Parry polynomials}
Let $w$ be a word with positive cumulative sign. The {\em Parry polynomial} of $w$, $P_w:\mathbb{C} \to \mathbb{C}$, is defined as 
\[P_w(z):=F(w, z)-1\]
 (cf. \cite[Definition 2.7]{BrayDavisLindseyWu}). It is evident that if $\It_\lambda=w^\infty$, then $\lambda$ is a root of $P_w$, and hence all Galois conjugates of $\lambda$ must be roots of $P_w$.

One can check by simple bookkeeping that for any word $w$ of positive cumulative sign, $P_w(z)$, $G(\r(w)^\infty, z)$ and $H(w^\infty, z)$ satisfy the following relationship:

\begin{lemma}\label{G-H-P}
If $w$ is of length $n$ and has positive cumulative sign, then
\[P_w(z)=(1-z^n)G\left(\r(w)^\infty, z \right)=z^n(1-z^{-n})H(w^\infty,z).\]
\qed
\end{lemma}

\subsection{Admissibility, itineraries and dominance}

The \emph{shift map} $\sigma$ is defined on sequences by
\[\sigma(w_1w_2w_3\ldots) = (w_2w_3\ldots)\ .\]

A sequence $w=w_1w_2\ldots$ is 
a generalized symbolic coding of $f_\lambda$ for some $\lambda \in (1,2]$ iff
$$f^k_{\lambda}(1) \in I_{w_{k+1},\lambda}$$ for every integer $k \geq 0$.  
\color{black}
Because the point $1/\lambda$ belongs to both intervals $I_{0,\lambda}$ and $I_{1,\lambda}$, there may exist more than one generalized symbolic coding for the itinerary of the point $1$ under $f_{\lambda}$. The {\em $\lambda$-itinerary} $\It_{\lambda}$  is the least (with respect to $\leq_E$) such generalized symbolic coding. 

A sequence $w$ starting with $10$ is called {\em admissible} if $$\sigma^k(w)\le_E w$$ for all $k \in \mathbb{N}$. A word $w$ is called {\em admissible} if $w$ has positive cumulative sign and $w^\infty$ is admissible.  

We will use the following immediate consequence of Theorem 12.1 of \cite{MilnorThurston}   \begin{theorem} \label{th:realizableadmissible}
 For every $\lambda \in (1,2]$, $\It_{\lambda}$ is admissible. 
  \end{theorem}
\color{black}

\begin{proposition}[\cite{BrayDavisLindseyWu}, Proposition 2.10] \label{prop:achievedasitinerary}
Let $w$ be a word with positive cumulative sign. If $w$ is admissible and the associated Parry polynomial, $P_w(z)$, can be written as the product of $(z-1)$ and another irreducible factor, then $w^{\infty} = \It_{\lambda}$ for some $\lambda \in (1,2]$. \end{proposition}

The following is a straightforward corollary of theorems of Milnor and Thurston (\cite{MilnorThurston}): 
\begin{corollary} \label{cor:monotonicity}
If $1 < \lambda < \lambda' \leq 2$, then $\It_{\lambda}<_E \It_{\lambda'}$.
\end{corollary}

 A word $w$ is called {\em dominant} (cf. \cite[Definition 4.1, Lemma 4.2]{BrayDavisLindseyWu}) if it has positive cumulative sign, and for any $1\leq k\leq |w|-1$,
 $$\suf_k(w)\cdot 1<_E\pr_{k+1}(w.)$$ Every dominant word is admissible, but admissible words may not be dominant.   A key property of the dominant words is the following, which is proved in \cite{TiozzoTopologicalEntropy}, and reviewed in \cite[Proposition 4.4]{BrayDavisLindseyWu}:

\begin{proposition}\label{prop:dense}
If $\lambda\in (\sqrt{2}, 2)$ and $\It_\lambda=w^\infty$, then for any $n>0$, there exists a word $w'$ such that $w^nw'$ is dominant.\qed

\end{proposition}

\subsection{Growth rates and critically periodic tent maps}
When a continuous self-map $f$ of an interval is postcritically finite, the exponential of its topological entropy, $e^{h_{top}(f)}$, also called its \emph{growth rate}, is a \emph{weak Perron number} -- a real positive algebraic integer whose modulus is greater than or equal to that of all of its Galois conjugates.  This is because cutting the interval at the critical and postcritical sets yields a Markov partition; each of the resulting subintervals is mapped to a finite union of subintervals.   The leading eigenvalue of the associated incidence matrix is $e^{h_{top}(f)}$, which the Perron-Frobenius Theorem implies is a weak Perron number.

In the present work, we consider growth rates of critically periodic unimodal interval self-maps.  A unimodal map $f$ is said to be \emph{critically periodic} if, denoting the critical point of $f$ by $c$, there exists $n \in \mathbb{N}$ such that $f^n(c) = c$.  
A theorem of Milnor and Thurston (\cite[Theorem 7.4]{MilnorThurston}) tells us that, from the point of view of entropy, instead of considering all critically periodic unimodal maps, we only need to consider critically periodic tent maps. For tent maps, it is easy to see that the growth rate is just the slope $\lambda$.



\subsection{Renormalization and doubling} 

As shown in \cite[Section 3]{BrayDavisLindseyWu}, for any $1< \lambda<\sqrt{2}$, the tent map $f_\lambda$ is critically periodic if and only if the tent map $f_{\lambda^2}$ is critically periodic. (This phenomenon is related to renormalization of the Mandelbrot set.) Furthermore, whenever $1<\lambda<\sqrt{2}$, $\It_{\lambda}$ can be obtained from $\It_{\lambda^2}$ by replacing each $1$ in $\It_{\lambda}$ with $10$ and each $0$ in $\It_{\lambda}$ with $11$.
 That is, the \emph{doubling map} $\mathfrak{D}:\{0,1\}^n \to \{0,1\}^{2n}$, $n \in \mathbb{N} \cup \{\infty\}$, defined by  
\[\mathfrak{D}(w_1 w_2 \ldots ) = 1 \cdot (w_1 + 1 \bmod 2) \cdot 1 \cdot (w_2 + 1 \bmod 2) \cdot  \ldots \]
satisfies $\mathfrak{D}(\It_{\lambda^2}) = \It_{\lambda}$ whenever $f_{\lambda}$ with $1 < \lambda < 2$ is critically periodic.
We say that a sequence $w$ is \emph{renormalizable} if there exists a sequence $w'$ such that $w=\mathfrak{D}(w')$; in this case we say that $w$ is the \emph{doubling} of $w'$ and call $w'$ the \emph{renormalization} of $w$.  We define renormalizable, doubling and renormalization for words analogously.

\section{Properties of the doubling map} \label{sec:renormalizationdoubling}

The goal of this section is to prove some elementary properties of renormalizable words and sequences that we will use in later sections to extend results about the part of the teapot above height $\sqrt{2}$ to the lower part.

\begin{lemma} \label{lem:DpreservesOrder}

The doubling map $\mathfrak{D}$ preserves the twisted lexicographic ordering $\leq_E$, cumulative signs, and hence also  admissibility.
\end{lemma}

\begin{proof}
If the number of $1$s in a word $w$ equals $n$, then for any letter $a$, the number of $1$s in $\pr_{2|w|+1}(\mathfrak{D}(w\cdot a))$ equals $2|w|+1-n$. It follows that if $n$ is odd, $w\cdot 1 <_E w \cdot 0$ and $\mathfrak{D}(w\cdot 1) <_E \mathfrak{D}(w\cdot 0)$; if $n$ is even, $w\cdot 0 <_E w\cdot 1$ and $\mathfrak{D}(w\cdot 0) <_E \mathfrak{D}(w\cdot 1)$.  Thus $\mathfrak{D}$ preserves $\leq_E$.  Furthermore, if a word $w$ has positive cumulative sign, then the number, $n$, of $1$'s in $w$ is even, implying that $\mathfrak{D}(w)$, which contains $2w-n$ $1$s, also has positive cumulative sign.  
 \end{proof}

\begin{lemma}\label{lem:DpreservesIter}
The doubling map $\mathfrak{D}$ takes itineraries to itineraries.  That is, if $\lambda^{2^k} = \lambda'$, then $\mathfrak{D}^k(\It_{\lambda'}) = \It_{\lambda}$. 
\end{lemma}

\begin{proof} By induction, it is easy to see that we only need to prove it for $k=1$, i.e. $\mathfrak{D}(\It_{\lambda^2})=\It_\lambda$. For any $\lambda\leq \sqrt{2}$, the tent map $f_{\lambda}$ sends the interval $[2/(\lambda+1), 1]$ to $[2-\lambda, 2/(\lambda+1)]$ and vice versa. Hence $f^2_{\lambda}$ is a tent map from $[2/(\lambda+1), 1]$ of slope $\lambda^2$, and any $x=f^{2k}_{\lambda}(1)$ lies on the left hand side of the critical point of $f^2_{\lambda}$ if and only if $x$ and $f_{\lambda}(x)$ are both to the right of $1/\lambda$, while $f^{2k}_{\lambda}(1)$ lies on the left hand side of the critical point of $f^2_{\lambda}$ if and only if $x$ is to the right of $1/\lambda$ and $f_{\lambda}(x)$ is to the right of $1/\lambda$, and this finishes the proof for the case when $\It_{\lambda^2}$ is not periodic. The case when $\It_{\lambda^2}$ is periodic follows from this argument together with Lemma~\ref{lem:DpreservesOrder}.
\end{proof}

\begin{proposition}\label{pro:renorm1} \
 If $w$ is a word with positive cumulative sign and $w'$ is the renormalization of $w$, then 
\[P_w(z)={z-1\over z+1}P_{w'}(z^2).\]
\end{proposition}

\begin{proof}
 Suppose $w$ and $w'$ are words satisfying $\mathfrak{D}(w')=w$. It is easy to see that if $w=w_1w_2\dots w_{2n}$ has positive cumulative sign, then $w'=w'_1w'_2\dots w'_n$ also has positive cumulative sign. So, (1) follows from the following more general statement: if $w'$ is any word, $w$ is the doubling of $w'$, then 
\begin{equation} \label{eq:Fequation} F(w, z)-1={z-1\over z+1}\left(F(w', z^2)-1\right).
\end{equation}
   We will prove \eqref{eq:Fequation} using induction on $|w'|$. In the base case $|w'|=1$, $w'=1$ or $w'=0$, and the statement is true by calculation.
    Now assume the statement is true for all words $w'$ such that  $|w'| \leq n-1$. Let $w'$ and $w$ be words with $|w'| =n$ and $\mathfrak{D}(w')=w$.  Let $w'_0$ be $w'$ with the last letter removed, and let $w_0$ be $w$ with the last two letters removed.  Then by the inductive hypothesis, 
  \[F(w_0, z)-1={z-1\over z+1}\left(F(w'_0, z^2)-1\right).\]
  We divide the inductive step into two cases:
  \begin{itemize}
  \item Case 1: $w'_n=0$. This implies $w=w_0\cdot 11$, so
    \begin{align*}
      F(w, z)-1 &=2-z \left(2-z(F(w_0, z))\right)-1\\
                &=2-z \left(2-z \left({z-1\over z+1}(F(w'_0, z^2)-1)+1 \right)\right)-1\\
                &={z-1\over z+1}\left(z^2F(w'_0, z^2)-1\right)\\
                &={z-1\over z+1}\left(F(w', z^2)-1\right) 
    \end{align*}
  \item Case 2: $w'_n=1$. This implies $w=w_0\cdot 10$, so
    \begin{align*}
      F(w, z)-1 &=z\left(2-z(F(w_0, z)) \right)-1\\
                &=z \left(2-z \left({z-1\over z+1}(F(w'_0, z^2)-1)+1\right)\right)-1\\
                &={z-1\over z+1} \left(2-z^2F(w'_0, z^2)-1\right)\\
                &={z-1\over z+1}\left(F(w', z^2)-1\right) 
    \end{align*}
  \end{itemize}

\end{proof}

\begin{proposition}\label{pro:renorm2} Let $w$ be an admissible word.  Then $w^{\infty}$ renormalizable if only if
 \[w^\infty<_E \It_{\sqrt{2}} \quad (= 10\cdot 1^\infty).\]
\end{proposition}

\begin{proof}
Firstly, it is easy to see that a sequence is renormalizable if and only if all its odd index letters are $1$, and a word is renormalizable if and only if it has even length and all its odd indexed letters are $1$. Because any admissible word starts with $10$, an admissible word $w$ is renormalizable if and only if $w^\infty$ is admissible and renormalizable.

Now suppose $w^\infty$ is admissible and renormalizable. Suppose the second $0$ in $w^\infty$ is at the $k^{\textrm{th}}$ location. It suffices to show that $\pr_{k-1}(w^\infty)$ has positive cumulative sign, which is equivalent to showing that $k$ is even, because the $(k-1)$-prefix of $w^\infty$ and $10\cdot 1^\infty$ are the same. This is an immediate consequence of the admissibility of $w^\infty$.

Now we prove the other direction. The sequence $w^\infty$ being admissible implies that the first $0$ in $w^\infty$ is at the second location. If we can further prove that the distance between any two consecutive $0$s is even, then all $0$s are at even locations, hence $w^\infty$ is admissible. Denote by $i_k$ the location of the $k^{\textrm{th}}$ $0$.  Let $k_m$ is the smallest number such that $i_{k_m}-i_{k_{m-1}}$ is odd.  Then by definition of $<_E$, 
\[\sigma^{i_{k_{m-1}}-1}(w^\infty)>_E10\cdot 1^\infty.\] 
\end{proof}

\begin{remark}\label{rem:iter_renorm}
  By $k^{\textrm{th}}$ renormalization or $k^{\textrm{th}}$ doubling, we mean carrying out the renormalization or doubling on a word or sequence $k$ times.  Proposition~\ref{pro:renorm1} above implies that if $w'$ is the $k^{\textrm{th}}$ renormalization of $w$, then the roots of $P_w$ not on the unit circle are the $(2^k)^{\textrm{th}}$ roots of the roots of $P_{w'}$ that are not on the unit circle.
  
Furthermore, because renormalization of sequences preserves $<_E$ (Lemma \ref{lem:DpreservesOrder}), we can apply part Proposition~\ref{pro:renorm2} above repeatedly to show that if the $w_k$ is the $k^{\textrm{th}}$ doubling of $10\cdot 1^\infty$, $w$ is admissible and $w^\infty<_Ew_k$, then $w$ has a $k^{\textrm{th}}$ renormalization.
\end{remark}

\section{Roots in $\mathbb{D}$ of reducible Parry polynomials} \label{s:reducibleParrypolys}

The purpose of this section is to prove Theorem \ref{t:InsideParryConjugatesDontMatter}, an alternative characterization of sets $\X_\lambda\cap \overline{\mathbb{D}}$, for $\lambda \in (1,2]$, using the results in \cite{BrayDavisLindseyWu}.  An upshot of Theorem \ref{t:InsideParryConjugatesDontMatter} is that we do not need to worry about extraneous roots in $\mathbb{D}$ from reducible Parry polynomials. 

We will use the following four results from \cite{BrayDavisLindseyWu}:

\begin{theorem}\cite[Theorem 1  (``Persistence Theorem''), Theorem 2]{BrayDavisLindseyWu}\label{thm:persist}
If $(z, \lambda)\in\U$, $|z|\leq 1$, then so is $(z, y)$ for any $y\in [\lambda, 2]$.
\end{theorem}

\begin{proposition}\cite[Lemma 5.3]{BrayDavisLindseyWu}\label{lem:concat}
 Let $w_1$ be dominant, $w_1>_E10\cdot 1^{|w_1|-2}$, $w_2$ be admissible, $w_1^\infty>_Ew_2^\infty$, and assume that there is some $m$ such that
  \[2m|w_2|>|w_1|>m|w_2|.\] Then there is some $w'$, some integer $m'\geq m$, such that $(w_1w'w_2^{m'})^\infty$ is admissible, 
  \[|w_1|+|w'|\geq m'|w_2|,\] and the Parry polynomial  $P_{w_1w'w_2^{m'}}(z)$ can be written as the product of  $(z-1)$ and another polynomial $Q(z)$ such that $Q(z^{2^k})$ is irreducible for all integers $k \geq 0$. 
\end{proposition}

\begin{proposition}\cite[Lemma 5.5]{BrayDavisLindseyWu}\label{lem:root_approx}
If $w_2$ is an admissible word and $z\in\mathbb{D}$ is a root of $P_{w_2}$, then for any $\epsilon>0$ there exists $N\in \mathbb{N}$ such that for any word $w_1$ and any integer $n\ge N$, $P_{w_1w_2^n}$ has a root within distance $\epsilon$ of $z$. 
\end{proposition}

\begin{proposition}\cite[Lemma 5.7, Remark 5.8]{BrayDavisLindseyWu}\label{lem:lead_root_approx}
If $y\in [\sqrt{2}, 2]$, for any $\epsilon>0$, there exists a dominant word $w_1$ such that for any word $w_2$, the leading root of $P_{w_1w_2}$ is within distance $\epsilon$ of $y$, and  $w_1>_E10\cdot 1^{|w_1|-2}$.
\end{proposition}

 We now use the above results to establish the following characterization of the sets $\X_\lambda\cap \overline{\mathbb{D}}$, which will be the starting point of our proof of Theorem \ref{t:improvedinsidecharacterization}.

\begin{theorem} \label{t:InsideParryConjugatesDontMatter}
 Fix  $1 < \lambda < 2$.  For each $\lambda' > \lambda$, define  $Y_{\lambda'}$ to be the closure of the set of roots in $\overline{\mathbb{D}}$ of all Parry polynomials $P_w$ such that $w$ is admissible and  $w^\infty \leq_E\It_{\lambda'}$, union with $S^1$, i.e. 
  $$Y_{\lambda'} := S^1 \cup \overline{ \{z \in \overline{\mathbb{D}} : P_w(z) = 0 \textrm{ for some admissible word } w \textrm{ such that } w^\infty \leq_E\It_{\lambda'}\} }.$$
  Then 
$$\X_{\lambda}\cap \overline{\mathbb{D}} = \bigcap_{\lambda' > \lambda} Y_{\lambda'}.$$
 

\end{theorem}

 \begin{remark}
The condition ``$w^\infty<_E\It_{\lambda'}$ for every $\lambda'>\lambda$'' is different from ``$w^{\infty} \leq \It_{\lambda}$'' because there could exist a symbolic coding for the itinerary of $1$ under the tent map $f_{\lambda}$ that is $>_E \It_{\lambda}$.  
\end{remark}

\begin{proof}  For any $1 < \lambda < 2$,
let 
$$\X'_\lambda = \bigcap_{\lambda' > \lambda} Y_{\lambda'}.$$
We will first prove $\X_{\lambda} \subseteq \X_{\lambda}'$. 
For any $\lambda'$, define the set $Z_{\lambda'}$ to be the closure of the set of Galois conjugates of critically periodic growth rates that are at most $\lambda'$, union with $S^1$.  By the Persistence Theorem, $\lambda_1 < \lambda_2$ implies $Z_{\lambda_1} \subseteq Z_{\lambda_2}$.  So if any point $x \in \bigcap_{\lambda' > \lambda} Z_{\lambda}$, then $x \in \X_{\lambda'}$ since $\Upsilon_2^{cp}$ is closed; similarly, if $x \not \in  \bigcap_{\lambda' > \lambda} Z_{\lambda}$, then $x \not \in \X_{\lambda}$.  Hence 
$$\X_{\lambda}\cap\overline{\mathbb{D}} = \bigcap_{\lambda' > \lambda} Z_{\lambda'}.$$
The conclusion will now follow from the statement that $Z_{\lambda'} \subseteq Y_{\lambda'}$ for all $\lambda'$.  If $z$ is a Galois conjugate of a critically periodic growth rate $\lambda''$ that is at most $\lambda'$, then $z$ is a root of the Parry polynomial $P_w$ such that $w^{\infty}=\It_{\lambda''}$, and $\It_{\lambda''} \leq_E \It_{\lambda'}$ by Corollary \ref{cor:monotonicity}.  Thus, $Z_{\lambda'} \subseteq Y_{\lambda'}$ for all $\lambda'$.


We will now prove $\X'_\lambda \subseteq \X_\lambda$.  To do this, it suffices to show $$Y_{\lambda'} \subseteq \bigcap_{\lambda'' > \lambda'} Z_{\lambda''}.$$
We first consider the case $\lambda'\geq \sqrt{2}$.
 Suppose $z$ is the root of some $P_w$, where $w$ is admissible and the leading root of $P_w$ is no larger than $\lambda'$.  ($Y_{\lambda'}$ is the closure of all such $z$'s).  For any $\epsilon>0$, Proposition \ref{lem:lead_root_approx} guarantees the existence of a  dominant word $w_1$ such that for any $w_2$, $P_{w_1w_2}$ is in $[\lambda', \lambda'+\epsilon)$ and $w_1>_E10\cdot 1^{|w_1|-2}$.  By monotonicity (Corollary \ref{cor:monotonicity}), $w_1^{\infty} >_E w^{\infty}$. 
Without loss of generality, we may choose $w_1$ so that its length, $|w_1|$, is arbitrarily big (this is because as we let $\epsilon \to 0$, we get arbitrarily many such dominant strings, and there are finitely many strings of at most any given length).   Thus we may assume that $w_1$ and $w$ satisfy the assumptions of Proposition~\ref{lem:concat} with the $m$ of Proposition~\ref{lem:concat} being arbitrarily large, and in particular, $m$ is $\geq$ the $N$ of Proposition ~\ref{lem:root_approx} using $w$ for $w_2$.  Let $w_3$ be the word constructed by Proposition~\ref{lem:concat}.   Because $w_3$ is admissible, has positive cumulative sign, and $P_{w_3}(z)/(z-1)$ is irreducible, $w_3^{\infty} = \It_{\lambda_3}$ for some $\lambda_3$ by Proposition \ref{prop:achievedasitinerary}.  We know $\lambda_3 \in [\lambda', \lambda'+\epsilon]$ because $w_3$ has the prefix $w_1$.  Also, any root of $P_{w_3}$ in $\mathbb{D}$ will be a Galois conjugate of $\lambda_3$, and by construction $P_{w_3}$ has a root close to $z$.  The containment now follows from letting $\epsilon\rightarrow 0$.

 

  Now we deal with the case $1<\lambda'<\sqrt{2}$. Let $k$ be the unique natural number such that $(\lambda')^{2^k}\in [\sqrt{2}, 2)$. Remark~\ref{rem:iter_renorm} implies that $w$ has a $k^{\textrm{th}}$ renormalization $w_0$, and $z^{2^k}$ is a root of $P_{w_0}$. Using $w_0$ in place of $w$ in the argument in the previous paragraph, we get a critically periodic growth rate $\lambda_4$ close to $(\lambda')^{2^k}$, such that one of its Galois conjugates $z_2$ is close to $z^{2^k}$. The conclusion in Proposition~\ref{lem:concat} further implies that any $(2^k)^{\textrm{th}}$ root of $z_2$ must be a Galois conjugate of the $(2^k)^{\textrm{th}}$ root of $\lambda_4$ as well, which implies that there is a Galois conjugate of $\lambda_4^{2^{-k}}$ which is close to $z$, which finishes the proof of the proposition.
\end{proof}

The following corollary is not used to prove any further results in the present work. 

\begin{corollary} \label{rem_irr_dense}
Let $V$ denote the set of all real numbers $\lambda \in (1,2)$ such that 
\begin{enumerate}
\item  the tent map $f_{\lambda}$ is critically periodic, 
\item there exists a word $w$ such that $\It_{\lambda} = w^{\infty}$,
\item the Parry polynomial $P_w(z)$ can be written as the product of an irreducible polynomial (in $\mathbb{Z}[z]$) and some cyclotomic polynomials.  
\end{enumerate}
Then $V$ is dense in $[1,2]$. 
\end{corollary}

\begin{proof}
The growth rates $\lambda_3$, as well as the growth rates $\lambda_3^{2^{-k}}$, $k \in \mathbb{N}$, constructed in the  proof of Theorem \ref{t:InsideParryConjugatesDontMatter} all satisfy conditions (1)-(3). 
\end{proof}

\section{$\lambda$-suitability} \label{sec:lambdasuitability}

In this section, we establish some basic properties of $\lambda$-suitability and prove the technical lemmas about $\lambda$-suitability that we will need in Section  \ref{s:insidecylinder}.

For convenience, we reproduce the definition of $\lambda$-suitability here: 
  For $\lambda \in (1,2)$, a sequence $w$ is called \emph{$\lambda$-suitable} if for every  $\lambda'\in (\lambda, 2]$, the following conditions hold:
    \begin{enumerate}
\item    $\r ( \pr_n(w)) \leq_E \pr_n(\It_{\lambda'})$ for all $n \in \mathbb{N}$.
    \item   If $\r(\pr_n(w))  = \pr_n(\It_{\lambda'})$, then the cumulative sign $s(\pr_n(w)) = -1$.  
    \item  If $\It_{\lambda'} = 1 \cdot 0^k \cdot 1 \dots$,  $k \in \mathbb{N}$, then $w$ does not contain $k+1$ consecutive $0$s. \newline (That is, if $\It_{\lambda'}$ starts with $1$ followed by $k$ $0$s and then $1$, writing $w$ as $w=w_1w_2\ldots,$ there does not exist $n \in \mathbb{N}$ such that $w_i = 0$ for all $n \leq i \leq n+k$.) 
     \item If $n \in \mathbb{N}$ satisfies $\sqrt{2} \leq_E (\lambda')^{2^n} < 2$, then $w = \mathfrak{D}'^n(w')$ for some sequence $w'$, where $\mathfrak{D}'$ is the map that replaces $0$ with $11$ and $1$ with $01$.  
    Furthermore, if $$\It_{\lambda'^{2^n}} = 1\cdot 0^k \cdot1\ldots,$$ then $w'$ does not contain $k+1$ consecutive $0$s.

     \end{enumerate} 
    \medskip

The  intuition behind the definition of $\lambda$-suitability is that we need a condition on sequences $w$ so that Lemma \ref{lem:main} works.

\begin{remark}
An immediate consequence of monotonicity (Corollary \ref{cor:monotonicity}) is that if $\lambda'$ satisfies conditions (1)-(4) of Definition \ref{def:improvedlambdasuitability} for a sequence $w$, then so does every $\lambda'' > \lambda'$.
\end{remark}

\begin{remark}
Every itinerary $\It_{\lambda'}$ is admissible (by Theorem \ref{th:realizableadmissible}), so the admissibility condition implies that if $\It_{\lambda'} = 1 \cdot 0^k \cdot 1 \dots$, then $\It_{\lambda}$ does not contain $k+1$ consecutive $0$s.  
\end{remark}

\begin{remark} Note that the map $\frak{D}'$ defined in the definition of $\lambda$-suitability is related to the doubling map $\frak{D}$ by 
$$\r \circ \pr_{2n} \circ \frak{D} = \frak{D}' \circ \r  \circ \pr_n (w)$$
for every sequence $w$ and $n \in \mathbb{N}$.
\end{remark}

\begin{lemma} \label{lem:lambdasuitableseqsclosed}
The set of $\lambda$-suitable sequences is closed. 
\end{lemma}

\begin{proof}
We will show that the set of all sequences that are not $\lambda$-suitable is open.   To do this, it suffices to show that given any sequence $w$ which is not $\lambda$-suitable, we can find a prefix of $w$ such that every sequence that shares this prefix is not $\lambda$-suitable. 
It is clear that conditions (1) and (2) are closed conditions.  For condition (3), we choose the prefix to be one that contains the first $k+1$ consecutive $0$s. Condition (4) is similar.  
\end{proof}

The following lemma is immediate because the definition of $\lambda$-suitability is of the form ``for all $\lambda' > \lambda$, $P(w,\lambda')$,'' where $P$ is a predicate.  

\begin{lemma}\label{lem:semicont}
  Let $\mathcal{M}_\lambda$ denote the set of $\lambda$-suitable sequences. Then 
$$\mathcal{M}_\lambda=\bigcap_{\lambda'' \in (\lambda,2]}\mathcal{M}_{\lambda''}.$$ \qed
\end{lemma}





\begin{lemma} \label{lem:reverseislambdasuitable}
If $w$ is an admissible word that satisfies $w^{\infty} \leq_E \It_{\lambda}$ for $\lambda \in (1,2)$, then $(\r(w))^{\infty}$ is $\lambda$-suitable. 
\end{lemma}

\begin{proof}
Observe that for any $n \in \mathbb{N}$
\begin{equation} \label{eq:somek}
\r ( \pr_n(\r(w)^{\infty})) = \pr_n(\sigma^{k}(w^{\infty}))
\end{equation} for some $k \in \mathbb{N}$, where $n+k$ is some multiple of $|w|$.  Since $w$ is admissible, $\pr_n(\sigma^k(w^{\infty}) \leq_E \pr_n(w^{\infty})$ for all $k,n \in \mathbb{N}$.  By Corollary \ref{cor:monotonicity}, for any $\lambda' > \lambda$, 
\begin{equation} \label{eq:lambdasstrictineq}
\It_{\lambda} <_E \It_{\lambda'}.
\end{equation}
We thus have that for any $n \in \mathbb{N}$, 
\begin{multline*}
\r ( \pr_n(\r(w)^{\infty})) = \pr_n(\sigma^{k}(w^{\infty})) \leq_E \pr_n(w^{\infty}) \\
 \leq_E  \pr_n(\It_{\lambda})  \leq_E \pr_n(\It_{\lambda'}),
\end{multline*}
which is condition (1) of the definition of $\lambda$-suitability. 

Now suppose that for some $\lambda' > \lambda$, 
$$\r ( \pr_n(\r(w)^{\infty}) = \pr_n(\It_{\lambda'})$$ and $\pr_n(\It_{\lambda'})$ has positive cumulative sign.  
Then from \eqref{eq:somek} we have $$\pr_n(\sigma^{k}(w^{\infty})) = \pr_n(\It_{\lambda'}).$$ 
Admissibility of $w$ and \eqref{eq:lambdasstrictineq} together imply that 
\begin{equation} \label{eq:gettingcommonprefix}
\sigma^{k}(w^{\infty})\leq_E w^{\infty}<_E\It_{\lambda'}.
\end{equation} 
Because $\pr_n(\It_{\lambda'})$ is the common prefix of $\sigma^{k}(w^{\infty})$ and $\It_{\lambda'}$,  \eqref{eq:gettingcommonprefix} implies it must also be a prefix of $w^\infty$. 
Removing this common $n$-prefix with positive cumulative sign from both sides of the inequality (by applying $\sigma^n$) yields 
$$w^\infty\leq_E\sigma^{n}(w^\infty).$$ However, admissibility also implies that $\sigma^{n}(w^\infty)\leq w^\infty$, so in fact 
$$w^\infty = \sigma^{n}(w^\infty).$$
Therefore 
\begin{equation} \label{eq:repeatedprefix}
w^\infty=(\pr_n(\It_{\lambda'}))^{\infty}.
\end{equation}

Let $j$ be the index of the first place $w^{\infty}$ differs from $\It_{\lambda'}$. Clearly, $j>n$. Pick $m \in \mathbb{N}$ such that $mn<j\leq (m+1)n$. Then, after removing the common prefix of length $mn$ and positive cumulative sign from both $w^{\infty}$ and $\It_{\lambda'}$, we get from \eqref{eq:repeatedprefix} and \eqref{eq:lambdasstrictineq} that
$$\sigma^{mn}(w^{\infty}) = w^{\infty}<_E \It_{\lambda'},$$
and hence
$$\pr_n(\It_{\lambda'})=\pr_n(w^\infty)<_E\pr_n(\sigma^{mn}(\It_{\lambda'})),$$ which contradicts with the fact that $\It_{\lambda'}$ is admissible (by Theorem \ref{th:realizableadmissible}). Thus, condition (2) of the definition of $\lambda$-suitability holds.

Now condition (3) of the definition of $\lambda$-suitability follows from the assumption that $w^{\infty} \leq_E \It_{\lambda}$. 

For condition (4), suppose for some $\lambda'>\lambda$, $\sqrt{2}\leq (\lambda')^{2^n}<2$. Then $\lambda^{2^n}<2$, so by Lemma~\ref{pro:renorm2}, $w=\mathfrak{D}^n(w')$ for some $w'$. Hence, 
$$(\r(w))^\infty=(\mathfrak{D}'^n(\r(w'))^\infty.$$ Because $\mathfrak{D}$ preserves $\leq_E$ and sends itineraries to itineraries (Lemma \ref{lem:DpreservesIter}), the number of consecutive $0$s in $(\r(w'))^\infty$, which is the number of consecutive $0$s in $w'^\infty$, can not be more than the number of consecutive $0$s in $\It_{\lambda'^{2^n}}$.


\end{proof}

The key combinatorial result we need to prove Theorem~\ref{t:improvedinsidecharacterization} is the following:

\begin{lemma}\label{lem:main} Fix $\lambda \in [1,2)$ and let $w_0$ be a finite dominant word such that $\It_{\lambda'}\leq_E w_0^\infty$ for some $\lambda'>\lambda$.  Let $\alpha$ be a word such that $\alpha$
\begin{enumerate}
\item  ends with $1$,
\item  is a prefix of some $\lambda$-suitable sequence,
\item has positive cumulative sign, and 
\item $|w_0|>|\alpha|$. 
\end{enumerate}
Then the word $w_0\cdot \r(\alpha)$ is admissible.
\end{lemma}

\begin{proof}
  Let $\alpha'=\r(\alpha)$. It suffices to show that the admissibility criterion 
  \[\sigma^k((w_0\alpha')^\infty)\leq_E (w_0\alpha')^\infty\]
   holds for all $1\leq k< |\alpha|+|w_0|$.

Case 1: $k<|w_0|$. This implies that the comparison between $\sigma^k((w_0\alpha')^\infty)$ and $(w_0\alpha')^\infty$ is equivalent to the comparison of a proper suffix of $w_0$ concatenated with $1$ with a prefix of $w_0$ of the same length.  Hence 
  \[\sigma^k((w_0\alpha')^\infty)\leq_E (w_0\alpha')^\infty\] because $w_0$ is dominant.
  
 Case 2: $|w_0|\leq k < |\alpha|+|w_0|$. Suppose the first place $\sigma^k((w_0\alpha')^\infty)$ and $(w_0\alpha')^\infty$ differ is at the $j^{\textrm{th}}$ position. It is evident that $1\leq j\leq |\alpha|+|w_0|$. We divide this into two subcases:
  
    \begin{itemize}
      \item Case 2A: $j\leq |w_0|+|\alpha|-k$. The fact that $|w_0|>|\alpha|$ and $k\geq |w_0|$ implies that $j\leq |w_0|$. Hence, the comparison between $\sigma^k((w_0\alpha')^\infty)$ and $(w_0\alpha')^\infty$ is equivalent to the comparison of  a proper suffix of $\alpha'$ with a prefix of $w_0$.  Hence, item (1) of Definition~\ref{def:improvedlambdasuitability} gives us 
      $$\sigma^k\left((w_0\alpha')^\infty\right)\leq_E (w_0\alpha')^\infty.$$
       \medskip
      
     \item Case 2B: $j>|w_0|+|\alpha|-k$.  The word $\beta :=\suf_{|w_0|+|\alpha|-k}(\alpha')$, which is a common prefix of $\sigma^k((w_0\alpha')^\infty)$ and $(w_0\alpha')^\infty$, is identical to a prefix of $w_0$, which is $\ge_E$ than a prefix of $\It_{\lambda'}$ for some $\lambda'>\lambda$. Hence, due to item (2) of Definition~\ref{def:improvedlambdasuitability}, $\beta$ has negative cumulative sign. Now, using the conclusion of Case 1, we have:
\[\sigma_k^{|w_0|+|\alpha|-k}(\sigma^k((w_0\alpha')^\infty))=(w_0\alpha')^\infty>_E \sigma_k^{|w_0|+|\alpha|-k}((w_0\alpha')^\infty)\]
Hence, 
$$\sigma^k((w_0\alpha')^\infty)\leq_E (w_0\alpha')^\infty$$ because $\beta$ has negative cumulative sign.
    \end{itemize}
\end{proof}

\begin{lemma} \label{lem:suitabilityrenormalization} Let $w$ and $w'$ be sequences, and let$\lambda \in (1,2)$ and $k \in \mathbb{N}$ satisfy $\sqrt{2} \leq \lambda^{2^k} < 2$.  If $w$ is $\lambda$-suitable  and $w = \mathfrak{D'}^k(w')$, then $w'$ is $\lambda^{2^k}$-suitable.
\end{lemma}

\begin{proof}
  By induction we only need to prove it for $k=1$.  Assume $w=\mathfrak{D'}(w')$ is $\lambda$-suitable, we will now show that $w'$ satisfies (1)-(4) of Definition~\ref{def:improvedlambdasuitability}. By definition, 
  $$\r(\mathfrak{D'}(v))=\mathfrak{D}(\r(v))$$ for any word $v$, so for any $\lambda'>\lambda$.
  \[\r(\pr_{2n}(\mathfrak{D'}(w'))=\r(\mathfrak{D'}(\pr_n(w')))\]
  \[=\mathfrak{D}(\r(\pr_n(w')))\leq \pr_{2n}(\It_{\lambda'})=\mathfrak{D}(\pr_n(\It_{\lambda'^2}))\]
  Hence (1) is true for $w'$ because of Lemma~\ref{lem:DpreservesOrder} and \ref{lem:DpreservesIter}. Condition (2) of Definition~\ref{def:improvedlambdasuitability} can be verified similarly. It is easy to see that $w$ satisfies (4) implies that $w'$ satisfies (4). Lastly, we will now show that $w$ satisfies (4) will imply $w'$ satisfies (3): if $\lambda^2\geq \sqrt{2}$, this follows from the statement of (4). If $\lambda^2<\sqrt{2}$, (4) implies that $w'=\mathfrak{D'}(w'')$ for some $w''$, which implies that $w'$ can never have more than one consecutive 0, hence it also satisfies (3).
\end{proof}


\section{Characterization inside the unit cylinder} \label{s:insidecylinder}

\begin{lemma} \label{lem:compactness}
Let $\mathcal{K}$ denote the space of compact subsets of $\mathbb{R}^3$ with the Hausdorff metric topology.  Given any compact subset $K$ of $\mathcal{K}$, the union of the elements of $K$ is a compact subset of $\mathbb{R}^3$. 

\begin{proof}
First, we claim there exists $R >0$ such that $k \subset \overline{B_R(0)}$ for all $k \in K$.  If this was not the case, then there exist $k_1$ and $k_2$ in $K$ such that $d_{\textrm{Haus}}(k_1,k_2)$ is arbitrarily large, contradicting the fact that $K$ is compact.  Thus the claim is true. 

Consider $K \times \overline{B_R(0)}$.  As a product of compact sets, it is compact.  Consider the subset $C \subseteq K$ such that $C$ consists of all pairs $(k,x)$ such that $x \in k$.  We claim $C$ is closed, and thus as a closed subset of a compact set, $C$ is compact.  To see this, we will show that $C$ is sequentially closed, i.e. if $(k_i,x_i)$ is a sequence in $C$ converging to $(k_{\infty},x_{\infty}) \in K \times \overline{B_R(0)}$, then $(k_{\infty},x_{\infty}) \in C$.  We have that $k_i \to k_{\infty}$ and $x_i \to x_{\infty}$, so suppose $x_{\infty} \not \in k_{\infty}$.  Since $k_{\infty}$ is a compact set, $x \not \in k_{\infty}$ implies there exists $\epsilon > 0$ such that $B_{\epsilon}(x)$ is contained in the complement of $k_{\infty}$.  This implies that $\liminf d_{\textrm{Haus}}(k_i,k_{\infty}) \geq \epsilon$, contradicting the fact that $k_i \to k_{\infty}$ in the Hausdorff metric.  
So we have a continuous map from $C$ to $\mathbb{R}^3$ sending $(k,x)$ to $x$.  The image under this map is compact.  
\end{proof}
\end{lemma}

The following two Lemmas, which we state without proof, are immediate consequences of Rouch\'e's theorem:

\begin{lemma} \label{l:HausdorffCont}
Let $A$ be the set of power series with bounded coefficients equipped with the product topology.  Let $\mathcal{C}$ be the set of compact subsets of $\mathbb{C}$ equipped with the Hausdorff topology.  Then the map $\rho:A \to \mathcal{C}$ defined by 
$$\rho(f) = S^1 \cup \{z \in \mathbb{D} : f(z) = 0\}$$
is continuous.  
\end{lemma}

\begin{lemma}\label{series_approx} 
Fix real numbers $M>0$, $0<r<1$, $\epsilon>0$. Suppose $\alpha$ is a power series whose coefficients are all bounded in absolute value by $M$.
Then there exists a real number $N = N(\alpha,r,\epsilon,M)$ such that for every power series $\beta$ whose coefficients are all bounded in absolute value by $M$ and whose first $N$ terms equal the first $N$ terms of $\alpha$,  for each root $z$ of $\alpha$ with $|z|<r$ there exists a root $z'$ of $\beta$ such that $|z-z'|<\epsilon$. \qed
  \end{lemma}

Now we prove the first main theorem:

\begin{proof}[Proof of Theorem \ref{t:improvedinsidecharacterization}]

For the reader's convenience, we reproduce  here the statement of Theorem \ref{t:improvedinsidecharacterization}: For any $\lambda \in (1,2]$,

\[\X_\lambda \cap \overline{\mathbb{D}} = S^1\cup \left \{z\in\mathbb{D}: G(w, z)=1\text{ for some }\lambda\text{-suitable sequence }w \right\}.\]

\medskip

By Remark \ref{rem:holdsfortoplevel}, the result holds for $\lambda=2$.  So fix $\lambda \in (1,2)$.  For brevity, let 
\[Z_\lambda:=\{z\in \mathbb{D}: G(w, z)=1\text{ for some }\lambda\text{-suitable sequence }w\}.\] First, we show that $S^1\cup Z_\lambda$ is compact. For each sequence $w$, the function from $\mathbb{D}$ to $\mathbb{C}$ given by $z \mapsto G(w, \cdot)-1$ is a power series with bounded coefficients.  Furthermore, the map from the set of sequences $w$ (with the product topology) to the set of power series (with the product topology on coefficients) given by $w \mapsto G(w,\cdot)-1$ is continuous.  Therefore, Lemma  \ref{l:HausdorffCont} implies that the map $\rho$  from the set of sequences with the product topology to $\mathcal{C}$, the set of compact subsets of $\mathbb{C}$ with the Hausdorff topology, given by
$$\rho(w) = S^1 \cup \{z \in \mathbb{D} : G(w,z) = 1\}$$ is continuous.  
 By Lemma \ref{lem:lambdasuitableseqsclosed},  the set of all $\lambda$-suitable sequences is closed (in the product topology on the set of sequences), and hence compact.  
Therefore, since $\rho$ is continuous, 
$$\{\rho(w) : w \textrm{ is } \lambda-\textrm{suitable}\}$$ is a compact subset of $\mathcal{C}$. 
 Hence,  Lemma \ref{lem:compactness} implies that
\[\bigcup_{w\text{ is }\lambda\text{-suitable}}\rho(w)\]
 is compact. 
But this set is precisely $S^1\cup Z_\lambda$, so we have shown $S^1\cup Z_\lambda$ is compact for any $\lambda \in [1,2]$.

Next, we show that 
\[\X_\lambda\cap \overline{\mathbb{D}}\subseteq S^1\cup Z_\lambda.\]
Theorem~\ref{t:InsideParryConjugatesDontMatter} shows that 
\begin{equation} \label{eq:contained1}
\X_\lambda\cap \overline{\mathbb{D}}=\bigcap_{\lambda'>\lambda}Y_{\lambda'},
\end{equation}
 where $Y_{\lambda'}$ is defined to be the closure of the set of roots in $\overline{\mathbb{D}}$ of all Parry polynomials $P_w$ such that $w$ is admissible and  $w^\infty \leq_E\It_{\lambda'}$, union with $S^1$.  For each such $w$ let $w_r$ be the sequence $$w_r :=(\r(w))^\infty.$$ 
 So fix $\lambda' > \lambda$ and consider any admissible word $w$ such that $w^{\infty} \leq_E \It_{\lambda'}$. 
By Lemma~\ref{G-H-P}, 
$$P_w(z) = (1-z^{|w|})G(w_r,z).$$  
 By Lemma \ref{lem:reverseislambdasuitable}, $w_r$ is $\lambda'$-suitable.  \color{black}
 Hence, all roots in $\overline{\mathbb{D}}$ of $P_w$ are in $S^1 \cup Z_{\lambda'}$.  Then, since $Z_{\lambda'}$ is closed,  we have that 
 \begin{equation} \label{eq:contained2}
 Y_{\lambda'}\subseteq S^1\cup Z_{\lambda'}.
 \end{equation}

 Now, combining \eqref{eq:contained1} and \eqref{eq:contained2} shows that for any point $z\in \X_\lambda\cap \mathbb{D}$, for each $n\in\mathbb{N}$, there exists a $(\lambda+{1\over n})$-suitable sequence $v_n$ such that $G(v_n, z)=1$. Let $v_\infty$ be an accumulation point of the set $\{v_n : n \in \mathbb{N}\}$.  By Lemma~\ref{lem:semicont}, the sequence $v_\infty$ is $\lambda$-suitable.  The continuity of $w\mapsto G(w, \cdot)$ implies that $G(v_\infty, z)=1$. Hence $\X_{\lambda}\cap\overline{\mathbb{D}}\subseteq S^1\cup  Z_{\lambda}$.

Lastly, we show that $S^1\cup Z_\lambda\subseteq \X_\lambda\cap \overline{\mathbb{D}}.$  We know from \cite{BrayDavisLindseyWu} that $S^1 \times [1,2] \subset \Upsilon_2^{cp}$.  Thus $S^1\subset \X_\lambda$, so it suffices to show that $Z_\lambda\subset \X_{\lambda}$.  Fix a point $z\in Z_\lambda$ and let $w$ be a $\lambda$-suitable sequence such that $G(w, z)=1$.  By condition (4) of Definition~\ref{def:improvedlambdasuitability}, there exists a sequence $w'$ such that $w=\mathfrak{D'}^k(w')$, and by Lemma \ref{lem:suitabilityrenormalization}, $w'$ is $\lambda^{2^k}$-suitable, and $\lambda^{2^k}\geq \sqrt{2}$. In particular, if $\lambda\geq \sqrt{2}$, we can let $k=0$ and $w'=w$.  As a consequence, there are infinitely many prefixes of $w'$ that end with $1$ and have positive cumulative sign.

For any $m \in \mathbb{N}$ such that $\pr_m(w')$ has positive cumulative sign and any word $w''$ with positive cumulative sign, it follows immediately from the definitions of a Parry polynomial and of $G$ that the first $m$ terms of the power series $G(w', z)-1$ and  $P_{w'' \cdot \r(\pr_m(w'))}(z)$ agree. Therefore, for any fixed $\epsilon_1>0$, by Lemma~\ref{series_approx} there exists $N \in \mathbb{N}$ such that $\pr_N(w')$ ends with $1$ and has positive cumulative sign, and 
for any word $w''$ with positive cumulative sign, there exists a point $z' \in  B_{\epsilon_1}(z^{2^k})$ such that 
\begin{equation} \label{eq:closeroot}
P_{w''\cdot \r(\pr_N(w'))}(z')=0.
\end{equation}

For any fixed $\lambda'$ satisfying $2> \lambda' > \lambda^{2^k}$, pick a critically periodic growth rate $\lambda'' \in (\lambda^{2^k},\lambda')$ and word $w_0$ with positive cumulative sign such that  $\It_{\lambda''}=w_0^{\infty}$. Since $\lambda'' < \lambda'$, for sufficiently large $n$, 
$$w_0^n <_E \pr_{n|w_0|}(\It_{\lambda'}).$$  Hence, by Proposition~\ref{prop:dense}, there exists $n \in \mathbb{N}$ and a word $w_1'$ such that the word 
$$w_1:= w_0^nw'_1$$
is dominant, $|w_1| > |w'|$, and 
\begin{equation} \label{eq:prefixineq}
w_1<_E\pr_{|w_1|}(\It_{\lambda'}).
\end{equation}

By Lemma~\ref{lem:main}, 
$$w_1 \cdot \r(\pr_N(w))'$$ is admissible.
  By \eqref{eq:closeroot}, 
  $$P_{w_1\cdot \r(\pr_N(w'))}$$ has a root within distance $\epsilon_1$ of $z^{2^k}$.  By \eqref{eq:prefixineq},
 \[(w_1\cdot \r(\pr_N(w')))^\infty<_E \It_{\lambda'}.\]
 Hence, the $k^{\textrm{th}}$ doubling of $w_1\cdot \r(\pr_N(w'))$, denoted as $w_d$, satisfies 
 \[w_d^\infty<_E\It_{(\lambda')^{1/2^k}}\] 
 and $P_{w_d}$ has leading root in $[\lambda, (\lambda')^{1/2^k}]$ and a root in $B_{\epsilon'_1}(z)$, where $\epsilon'_1$ is the diameter of the preimage of $B_{\epsilon'}(z^{2^k})$ under the map $z\mapsto z^{2^k}$.
 
 Now, since $\epsilon_1>0$ and $\lambda'>\lambda$ were arbitrary, and since $\U$ is closed, we obtain that $(z, \lambda)\in \U$, and hence $z\in \X_\lambda$.
\end{proof}

\section{Characterization outside the unit cylinder} \label{sec:outsidecharacterization}

The goal of this section is to prove Theorem \ref{t:improvedoutsidecharacterization}, a characterization of the part of the Master Teapot that is outside the unit cylinder. This follows largely from arguments in \cite{TiozzoGaloisConjugates}, but we will include a proof here for the sake of completeness.

The following proposition is essentially a restatement of \cite[Proposition 3.3]{TiozzoGaloisConjugates}:
\begin{proposition} \label{p:rootsmovecontinuously}
 The map $\Phi:(1, 2) \rightarrow \{\text{compact subsets of \ }\overline{\mathbb{D}}\}$ given by 
\[\Phi(\lambda)=S^1\cup \left \{z: H(\lambda, z^{-1})=0 \right\}\]
is continuous in the Hausdorff topology.  
\end{proposition}

\begin{proof}
We only need to show that it is continuous at every point $\lambda_0\in (1, 2)$. If $\It_{\lambda_0}$ is not periodic, the forward orbit of $1$ under $f_{\lambda_0}$ never hits $1/{\lambda_0}$, hence $\It: \lambda\mapsto \It_\lambda$ is continuous at $\lambda_0$. This is because for any cylinder set $[a_1,\ldots,a_j]$, the set 
$$\{\lambda_1 \in (1,2]: \pr_j(\It_{\lambda_1}) = a_1\ldots a_j\}$$ is open.  The continuity of $\Phi$ follows from the definition of $H$ (Definition \ref{def:ifs}) and Lemma \ref{series_approx}. 

If $\It_{\lambda_0}$ is periodic, let $w_0$ be the word of shortest length such that $\It_{\lambda_0}=w_0^\infty$, and let $w_0'$ be the word with the same length as $w_0$ such that $\pr_{|w_0|-1}(w_0) = \pr_{|w_0|-1} (w_0')$ but whose last digit differs from that of $w_0$. Then the proof of Lemma 12.2 in \cite{MilnorThurston} implies 
$$\lim_{\lambda\rightarrow\lambda_0^-}\It_\lambda=w_0^\infty$$ and $$\lim_{\lambda\rightarrow\lambda_0^+}\It_\lambda={w_0'}^\infty.$$ 
 However, a simple computation (which we leave to the reader) shows that $H(w_0^\infty, z^{-1})$ and $H({w_0'}^\infty, z^{-1})$ differ by cyclotomic factors, and hence have the same roots inside $\mathbb{D}$.
\end{proof}

\begin{proof}[Proof of Theorem \ref{t:improvedoutsidecharacterization}]
  For convenience of notation, set
  \[R_\lambda=\{z: H(\It_\lambda, z)=0\}\]
  Let $I$ be a small closed neighborhood of $\lambda$ in $(1, 2)$. To show Theorem \ref{t:improvedoutsidecharacterization}, we only need to show \[ \bigcup_{\lambda\in I} \left((\X_\lambda\backslash\mathbb{D})\cup S^1\right)=\bigcup_{\lambda\in I}\left((R_\lambda\backslash\mathbb{D})\cup S^1\right).\]
   The fact that the right hand side is compact is due to Proposition~\ref{p:rootsmovecontinuously}. Furthermore, due to Remark~\ref{series_approx}, a dense subset of the left hand side is dense in the right hand side, so they are identical.
\end{proof}

\section{Algorithms to test membership of $\X_\lambda$} \label{sec:membership}

In this section we will describe an algorithm to check if a point $z_0\in \mathbb{C}$ is in the complement of a slice $\X_\lambda$, for $\lambda\in (1, 2)$.

Firstly, if $\lambda<\sqrt{2}$, Theorems \ref{t:improvedinsidecharacterization} and \ref{t:improvedoutsidecharacterization} implies that $z\in \X_\lambda$ if and only if  $z^2\in \X_{\lambda^2}$, so we can always reduce the question to the case $\lambda\in [\sqrt{2}, 2)$.

\subsection{Testing $z_0$ with $|z_0| > 1$}

When $|z_0|>1$, Theorem \ref{t:improvedoutsidecharacterization} gives us a straightforward way to test if $z_0\not\in \X_\lambda$ --  calculating the first few terms of the power series $H(\It_\lambda, z^{-1})$, then checking if $z_0^{-1}$ is a root of this power series. More precisely, we have the following algorithm:

\medskip
\begin{algorithm}[H] \label{algorithm1}
  \For{$n>1$}{
    Calculate $\pr_{n+1}(\It_\lambda)$\;
    Find the polynomial $P_n$ which consists of the first $n$-terms of power series $H(\It_\lambda, z^{-1})$\;
    If $ \left|P_n(z_0^{-1})\right|>{2|z_0|^{-n}\over 1-|z_0|}$, then $z_0\not\in\X_\lambda$\;
  }
\caption{\label{alg:outer} Algorithm to verify that $|z_0|>1$ is not in $\X_\lambda$}
\end{algorithm}
\medskip

\begin{remark}
If instead of checking if $z_0\not\in\X_\lambda$, we want to see if an $\epsilon$-neighborhood of $z_0$ is contained in the complement of $\X_\lambda$, we can change the last line of Algorithm \ref{algorithm1} to make use of Rouch\'e's theorem.
\end{remark}

\subsection{Testing $z_0$ with $|z_0| < 1$.}

If $|z_0|<1$, a way to certify that $z_0\not\in\X_\lambda$ is by first finding the set of all words of length $N$ that satisfy Conditions (1)-(3) of Definition~\ref{def:improvedlambdasuitability} (Condition (4) is trivial because $\lambda\geq\sqrt{2}$), denoted as $\mathcal{M}_{N, \lambda}$, for each word $w=(w_1\dots w_N)\in \mathcal{M}_{N, \lambda}$, evaluating $f_{w_N, z_0}^{-1}\circ f_{w_{N-1}, z_0}^{-1}\dots f_{w_1, z_0}^{-1}(1)$ and checking that they are all sufficiently large. More precisely, the algorithm can be described as follows:
\medskip 

\begin{algorithm}[H]
  \For{$N>1$}{
    Let $\mathcal{M}_{N, \lambda}$ be the set of all words of length $N$ that satisfies Conditions (1)-(3) in Definition~\ref{def:improvedlambdasuitability}\;
    Let $flag\leftarrow False$\;
    \For{$w\leftarrow (w_1\dots w_N)\in\mathcal{M}_{N, \lambda}$}{
      \If{$f_{w_N, z_0}^{-1}\circ f_{w_{N-1}, z_0}^{-1}\dots f_{w_1, z_0}^{-1}(1)\leq {2\over 1-|z_0|}$}{
        $flag\leftarrow True$\;
        Break\;
        }
     }
    If $flag=False$, then $z_0\not\in\X_\lambda$\;
  }
\caption{\label{alg:inner} Algorithm to verify that $|z_0|<1$ is not in $\X_\lambda$, where $\lambda\in[\sqrt{2},2)$.}
\end{algorithm}
\medskip

The reason that Algorithm~\ref{alg:inner} is true is due to the following proposition:

\begin{proposition}\label{prop:alg2}
  Let $\lambda\in [\sqrt{2}, 2)$, and let $\mathcal{M}_{N, \lambda}$ be defined as in Algorithm~\ref{alg:inner}. Suppose $|z|<1$, then $z\not\in\X_\lambda$ if and only if there exists $N \in \mathbb{N}$ such that for every word $w=w_1\dots w_N\in\mathcal{M}_N$, 
  \[f_{w_N, z}^{-1}\circ \ldots \circ f_{w_1, z}^{-1}(1)\geq {2\over 1-|z|}+\epsilon.\]
\end{proposition}

\begin{proof}
First, we assume that there is some $N$ such that for every word $w=w_1\dots w_N\in\mathcal{M}_N$,
 \[f_{w_N, z}^{-1}\circ \ldots  \circ f_{w_1, z}^{-1}(1)\geq {2\over 1-|z|}+\epsilon\] and prove that $z\not\in \X_\lambda$.
  Suppose $z\in \X_\lambda$.  Then by Theorem~\ref{t:improvedinsidecharacterization}, there must be some $\lambda$-suitable sequence $v=v_1v_2\dots $ such that
  \[1=G(v, z)=\lim_{n\rightarrow\infty}f_{v_1, z} \circ\ldots \circ f_{v_n, z}(1)\]
  In other words, for any $\delta>0$, there is some $n>N$ such that
  \[ \left|f_{v_1, z}\circ\ldots \circ f_{v_n, z}(1)-1 \right|<\delta\]
 By the definition of $\mathcal{M}_N$, the word $v_1\dots v_N\in \mathcal{M}_N$. Let $u=f_{v_1, z}\circ \ldots \circ f_{v_n, z}(1)$. Then $|u-1|<\delta$. Because $f_{v_N, z}^{-1} \circ \ldots \circ f_{v_1, z}^{-1}$ is continuous, we can pick $\delta$ small enough such that 
 \[f_{v_N, z}^{-1}\circ \ldots \circ f_{v_1, z}^{-1}(u)>{2\over 1-|z|}.\]
  However,
 \[f_{v_N, z}^{-1}\circ \ldots \circ f_{v_1, z}^{-1}(u)=f_{v_{N+1}, z}\circ\ldots \circ f_{v_n, z}(1)\]
 By calculation, it is easy to verify that $1$ is in the disc 
 $$D_{2\over {1-|z|}}=\left \{z\in\mathbb{C}:|z|\leq {2\over {1-|z|}}\right \},$$ and both $f_{0, z}$ and $f_{1, z}$ send $D_{2\over {1-|z|}}$ to itself. Hence 
 \[ \left|f_{v_{N+1}, z}\circ\ldots \circ f_{v_n, z}(1) \right|\leq {2\over {1-|z|}},\] a contradiction.

 Now, for the other direction, we assume that for any $N \in\mathbb{N}$ there is some word $w=w_1\dots w_N\in\mathcal{M}_N$ such that 
 \[f_{w_N, z}^{-1}\circ \ldots \circ f_{w_1, z}^{-1}(1)\leq {2\over 1-|z|}\] 
 and prove that $z\not\in \X_\lambda$. Let $C_N$ be the set of sequences such that an $N$-prefix of it is in $\mathcal{M}_N$, and this $N$ prefix is of the form $w_1\dots w_N$ such that 
 $$f_{w_N, z}^{-1}\circ f_{w_{N-1}, z}^{-1}\dots f_{w_1, z}^{-1}(1)\leq {2\over 1-|z|}.$$ The fact that $f_{0, z}$ and $f_{1, z}$ both send $D_{2\over {1-|z|}}$ to itself implies that $C_{N+1}\subset C_{N}$, and all these sets are non empty and compact under the product topology, hence their intersection is non-empty. Let $w\in\bigcap_NC_N$, then $w$ is $\lambda$-suitable and it is easy to see that $G(w, z)=1$.
\end{proof}

Furthermore, we have an effective version of the Proposition~\ref{prop:alg2} above:

\begin{proposition}\label{p:effective}
Let $\lambda$, $z$, $N$ and $\epsilon$ as in Proposition~\ref{prop:alg2} above, ${1\over 2}<|z|<1$. Then for any $y\in \mathbb{C}$, if 
\[|y-z|<\min \left \{{1-|z|\over 2}, {(1-|z|)^2\epsilon\over 16}, |z|-{1\over 2}, {\epsilon\over N\cdot 2^{N+1}} \right \},\]
 then $y\not\in\X_\lambda$.
\end{proposition}

\begin{remark}
The assumption $|z|>{1\over 2}$ is a reasonable one because it is well known (cf. \cite{TiozzoGaloisConjugates}) that if $|z|<{1\over 2}$ then $z\not\in\X_\lambda$ for any $\lambda\in (1, 2)$.
\end{remark}

\begin{proof}
  It is easy to see that as long as $|y|<1$, 
  \[\left|{2\over 1-|z|}-{2\over 1-|y|}\right|<\epsilon/2,\]
   and for any $w=w_1\dots w_N\in\mathcal{M}_N$,
  \[ \left|f_{w_N, z}^{-1}\circ f_{w_{N-1}, z}^{-1}\dots f_{w_1, z}^{-1}(1)-f_{w_N, y}^{-1}\circ f_{w_{N-1}, y}^{-1}\dots f_{w_1, y}^{-1}(1)\right|<\epsilon/2\]
  then $y$ also satisfy the assumption in Proposition~\ref{prop:alg2}. The first condition, $|y|<1$, holds because $|y-z|<{1-|z|\over 2}$, which implies $|y|<{1+|z|\over 2}<1$. The second condition,
   \[\left|{2\over 1-|z|}-{2\over 1-|y|}\right|<\epsilon/2,\]
    holds because $|y|<{1+|z|\over 2}$ and $|y-z|<{(1-|z|)^2\epsilon\over 16}$. The third condition, 
  \[\left |f_{w_N, z}^{-1}\circ f_{w_{N-1}, z}^{-1}\dots f_{w_1, z}^{-1}(1)-f_{w_N, y}^{-1}\circ f_{w_{N-1}, y}^{-1}\dots f_{w_1, y}^{-1}(1) \right|<\epsilon/2,\]
holds because of the following argument: As a polynomial of ${1\over z}$, 
$$f_{w_N, z}^{-1}\circ f_{w_{N-1}, z}^{-1}\dots f_{w_1, z}^{-1}(1)$$ has degree $N$ and coefficients bounded between $-2$ and $2$, hence has its derivative bounded by $N2^{N-1}\cdot 2=N\cdot 2^N$ on the annulus $\{y\in\mathbb{C}:1\leq |y|\leq 2\}$. Because $|y-z|<|z|-{1\over 2}$, $y$ is inside this annulus, so this third condition follows from the assumption that $|y-z|<{\epsilon\over N\cdot 2^{N+1}}$ and the mean value theorem.  
  \end{proof}

\section{Asymmetry of $\X_\lambda$} \label{sec:asymmetry}

The following proposition is likely well-known to experts; we include the proof for completeness. 

\begin{proposition}  \label{p:ThurstonSetSymmetrical}
$\Omega_2^{cp} \cap \mathbb{D}$ is invariant under reflection across the real axis and across the imaginary axis. \end{proposition}

\begin{proof}
The set $\Omega_2^{cp} \cap \mathbb{D}$ is invariant under reflection across the real axis because Galois conjugates come in complex conjugate pairs. Tiozzo \cite{TiozzoGaloisConjugates} showed that $\Omega_2^{cp} \cap \mathbb{D}\backslash S^1$ is the set of all the roots in $\mathbb{D}$ of all power series with all coefficients in $\{\pm 1\}$.  So if $z \in \mathbb{D}$ is a root of a power series $S$ with coefficients in $\{\pm 1\}$, then $-z$ is a root of the power series formed from $S$ by flipping the sign of the coefficients on all terms of odd degree. Therefore the complex conjugate, $\overline{-z}$, is in $\Omega_2^{cp}$. 
 \end{proof}

However, our Algorithm~\ref{alg:inner} in the previous section can be used to show that $\X_\lambda\cap\mathbb{D}$ does not necessarily have such symmetry, which proves Theorem~\ref{t:teapotnotsymmetrical}:

\begin{proof}[Proof of Theorem \ref{t:teapotnotsymmetrical}]
We only need to show that there is some $z\in \X_{1.82}\cap\mathbb{D}$ such that $-z\not\in\X_{1.82}\cap\mathbb{D}$. Consider the tent map with growth rate being the leading root of 
\[-1 + z^2 - z^4 + z^6 - 2 z^7 + 3 z^8 - 4 z^9 + 3 z^{10} - 2 z^{11} + z^{12} - 2 z^{13} + z^{14},\] which is approximately $1.8149185987640513$ and is smaller than $1.82$, hence any Galois conjugate of this leading root must be in $\X_{1.82}$. Let $z$ be the Galois conjugate near the point $-0.5840341196392905+0.4820600149798202 i$. Applying Algorithm~\ref{alg:inner} to $-z$ for $N=20$ shows that $-z\not\in \X_{1.82}$.
\end{proof}

\bibliographystyle{alpha}
\bibliography{PCF2Bibliography.bib}

\noindent Kathryn Lindsey, Boston College,  Department of Mathematics, Maloney Hall, Fifth Floor, Chestnut Hill, MA, \mbox{\url{kathryn.a.lindsey@gmail.com}}

\noindent Chenxi Wu, Rutgers University, Department of Mathematics, 110 Frelinghuysen Road, Piscataway, NJ, \mbox{\url{wuchenxi2013@gmail.com}}

\end{document}